\documentclass[12pt, reqno]{amsart}
\usepackage{amsmath, amsthm, amscd, amsfonts, amssymb, graphicx, color}
\usepackage[bookmarksnumbered, colorlinks, plainpages]{hyperref}
\usepackage[numbers,sort&compress]{natbib}
\hypersetup{colorlinks=true,linkcolor=red, anchorcolor=green, citecolor=cyan, urlcolor=red, filecolor=magenta, pdftoolbar=true}

\newtheorem{theorem}{Theorem}[section]
\newtheorem{lemma}[theorem]{Lemma}
\newtheorem{proposition}[theorem]{Proposition}

\theoremstyle{definition}
\newtheorem{definition}[theorem]{Definition}

\theoremstyle{remark}
\newtheorem{remark}[theorem]{Remark}
\numberwithin{equation}{section}

\begin{document}

\setcounter{page}{1}

\title[A property in vector-valued function spaces]{A property in vector-valued function spaces}
\author[Kexin Zhao \MakeLowercase{and} Dongni Tan*]{Kexin Zhao $^1$  \MakeLowercase{and} Dongni Tan$^{2*}$}

\address{$^{1}$Department of Mathematics, Tianjin University of Technology, Tianjin 300384, P.R. China.}
\email{\textcolor[rgb]{0.00,0.00,0.84}{1435080446@qq.com}}

\address{$^{2}$Department of Mathematics, Tianjin University of Technology, Tianjin 300384, P.R. China.}
\email{\textcolor[rgb]{0.00,0.00,0.84}{tandongni0608@sina.cn}}
\let\thefootnote\relax\footnote{The second author is supported by the Natural Science Foundation of China (Grant Nos. 11371201, 11201337, 11201338, 11301384). \newline \indent $^{*}$Corresponding author}



\begin{abstract} This paper deals with a property which is equivalent to generalised-lushness for separable spaces. It thus may be seemed as a geometrical property of a Banach space which ensures the space to have the Mazur-Ulam property.
We prove that if a Banach space $X$ enjoys this property if and only if $C(K,X)$ enjoys this property. We also show the same result holds for  $L_\infty(\mu,X)$ and $L_1(\mu,X)$.
\end{abstract}  \maketitle

\section{\textbf{Introduction }}Let us first give some notation. For a Banach space $X$,  $B_X$, $S_X$ and $X^*$  will stand for its unit ball, its unit sphere
and its dual space, respectively.  All spaces are over the real
field. A slice is a subset of $B_X$ of the form
\begin{align*}
S(x^*,\alpha)=\{x\in B_X:  x^*(x)>1-\alpha\},
\end{align*}
where $x^*\in S_{X^*}$ and $0<\alpha<1$.
A topic now known as
{\it Tingley's problem} or {\it the isometric extension problem} was first raised by D. Tingley \cite{3}. It is described as follows: let $T$ be a surjective isometry between $S_X$ and $S_Y$. Is it true that $T$ extends to a linear isometry
$U : X\rightarrow Y$ of the corresponding spaces?

Although this problem for general spaces remains unsolved even in dimension two, there is a number of publications devoted to Tingley's problem (say,
Zentralblatt Math. shows 57 related papers published from 2002 to
2019). The positive answers for many classical Banach spaces were given in \cite{KM,THL} and the references
therein. It is well worth mentioning that there is a fruitful series of recent papers dealing
with Tingley's problem and related questions for operator algebras, for example, see \cite{BF,F1,F2}. The interested reader is referred to the survey \cite{Pe18} for more information on operator algebras, and for other recent contributions not considered in the survey, please see \cite{C,CA,KP,WH}.

The notion of the Mazur-Ulam property was introduced by Cheng and Dong in \cite{CD}:  a (real) Banach
space $X$ is said to have {\it the Mazur-Ulam property} (MUP) if for every
Banach space $Y$ every surjective isometry between $S_X$ and $S_Y$ extends to a real linear isometry from $X$ onto $Y$. Kadets and Mart\'{i}n \cite{KM} proved that all finite-dimensional polyhedral spaces (i.e. those spaces whose unit ball is a polyhedron) have the MUP.  In order to show that a large class of Banach spaces enjoy the MUP, Tan, Huang and Liu introduced in \cite{THL} the notion of generalized-lushness.
\begin{definition}
A Banach space $X$ is said to be generalized-lush (GL) if for every $x\in S_{X}$ and every $\varepsilon>0$, there exists a slice $S_{x^*}:=S(x^{*},\varepsilon)$ with $x^{*}\in S_{X^{*}}$ such that
\[x\in S_{x^*} \quad \mbox{and} \quad \mbox{dist}(y,S_{x^*})+\mbox{dist}(y,-S_{x^*})<2+\varepsilon \,\quad \mbox{for all}\,\ y\in S_{X}.\]
\end{definition}
This definition, at least for separable spaces, is a generalisation of the concept of lushness introduced in \cite{KVM} which has a connection with the numerical index of operators. For more spaces with MUP, the authors of \cite{THL} further introduced the concept of local-generalized-lushness.
\begin{definition}
A Banach space $X$ is said to be a local-GL-space if for every separable
subspace $E\subset X$, there is a GL-subspace $F\subset X$ such that $E\subset F\subset X$.
\end{definition}
 In \cite{THL}, it is shown that that all local-GL-spaces (and consequently
all GL-spaces, all lush spaces) possess the MUP. Moreover many stable properties for GL-spaces are established in
\cite{THL}, for example, it is established that the
class of GL-spaces is stable under $c_0$, $l_1$ and $l^\infty$-sums (\cite[Theorem 2.11 and Proposition 2.12]{THL}) and that if $X$ is a GL space then so is the space $C(K,X)$ of all continuous
functions from  any compact Hausdorff space $K$ into $X$ (\cite[Theorem 2.10]{THL}).
 Later Jan-David Hardtke in \cite{H} stated that a large class of
GL-spaces is stable under ultraproducts and under passing to a large
class of $F$-ideals, in particular to $M$-ideals. And more, he introduced in \cite{H} (with the help of an anonymous referee as is mentioned in the \cite[2.4 Lush spaces]{H2}) the following (at least formally) weaker version
of GL-spaces:
\begin{definition}\label{def1}
A Banach space $X$ is said to have the property $(**)$ if for all $x_1,x_2\in S_X$ and every $\varepsilon>0$, there exists a slice $S_{x^*}:=S(x^{*},\varepsilon)$ with $x^{*}\in S_{X^{*}}$ such that
\begin{equation}\label{equ:24}
x_1\in S_{x^*}\,\,\, \mbox{and}\,\, \mbox{dist}(x_2,S_{x^*})+\mbox{dist}(x_2,-S_{x^*})<2+\varepsilon.
\end{equation}
\end{definition}
 Throughout what follows, we shall
freely use without explicit mention an elementary fact that Definition \ref{def1} is equivalent to another one where the assumption: $x_1,x_2\in S_X$ is replaced by $x_1\in S_X$ and $x_2\in B_X$.
It should be remarked that the following observations were made in \cite{H}.
\begin{enumerate}
  \item Every lush space has the property ($**$).
  \item For separable spaces, ($**$) is equivalent to GL.
  \item  Every space with the property ($**$) has the MUP.
\end{enumerate}

Very recently, a stability results that $X$ having the property ($**$) implies that $L_1(\mu, X)$ and $L_\infty(\mu,X)$ also have the the property ($**$) with $(\Omega,\Sigma,\mu)$ being a $\sigma$-finite measure space  has been proved in \cite[Theorem 4.8]{H2} by a reduction theorem.  In fact, this reduction theorem is shown in \cite{H2} for a large class of spaces that enjoy a certain type of geometric properties, such as octahedrality, almost
squareness, lushness, the Daugavet property and so on. In the earlier time, stronger stability results for lushness have already been stated in recent monograph \cite{MRK}: $C(K,X)$ is lush if and only if $X$ is, and the same results hold for $L_1(\mu, X)$ and $L_\infty(\mu,X)$. The aim of this paper is to demonstrate that these results remain true for the property ($**$) in the same spaces.

Let us make a comment here on vector-valued function spaces for GL-spaces. We only know that if $X$ is a GL-space, then so are $C(K,X)$ (\cite[Theorem 2.10]{THL}) and $L_1(\mu,X)$ (\cite[Theorem 5.1]{H2}). It is not known whether this is true for $L_\infty(\mu, X)$ nor if $X$ is a GL-space whenever $ C(K,X)$, $L_1(\mu,X)$ or $L_\infty(\mu,X)$ is a GL-space, where $X$ is non-separable.

Throughout the paper, given a compact Hausdorff topological
space $K$ and a Banach space $X$, $C(K, X)$ is the Banach space of all continuous
functions from $K$ into $X$ endowed with the supremum norm.  Given a $\sigma$-finite
measure space  $(\Omega,\Sigma,\mu)$, for $A\in\Sigma$, $\chi_A$ is the characteristic function of $A$, and for a Banach space $X$,
$L_\infty(\mu, X)$ is the Banach space of
all (clases of) measurable functions $f$ from $\Omega$ into $X$ which are essentially bounded,
endowed with the essential supremum norm
\begin{center}
$\|f\|_{\infty}=$ess sup$\{\|f(t):t\in\Omega\}$.
\end{center}
$L_1(\mu, X)$ is the Banach space of all
(clases of) Bochner-integrable functions from $\Omega$ into $X$, endowed with the integral
norm
 $$ \|f\|_1=\int_\Omega \|f(t)\|du(t).$$

 \section{the results}

Our aim is to present several results concerning the property ($**$) for vector-valued function spaces.
We begin this with the spaces of continuous functions. The proof of the  ``only if'' part of the following result is an easy adaptation of \cite[Theorem 2.10]{THL}. We present it here for completeness.
\begin{theorem}\label{CK-theorem}
Let $K$ be a compact Hausdorff topological space, and let $X$ be a Banach space. Then $X$ has the property ($**$) if and only if $C(K,X)$ has the property ($**$).
\end{theorem}
 \begin{proof}
We first show the ``only if" part.  Let $f_{1},f_2 \in S_{C(K,X)}$ and $\varepsilon>0$. It is clear that there exists a $t_{0}\in K$ such that $\|f_{1}(t_{0})\|=1$.  Since $X$ has  the property ($**$), it follows that there exists a slice $S_{x^*}:=S(x^{*},\frac{\varepsilon}{4})$ with $x^{*}\in S_{X^{*}}$ such that $f_{1}(t_{0})\in S_{X^{*}}$ and \[\mbox{dist}\left(f_{2}(t_{0}),S_{x^{*}}\right)+\mbox{dist}\left(f_{2}(t_{0}),-S_{x^{*}}\right)<2+\frac{\varepsilon}{4}.\]
 Namely, we can find $y_{1}\in S_{x^{*}}$ and $y_{2}\in -S_{x^{*}}$ such that \[\left\|f_{2}(t_{0})-y_1\right\|+\left\|f_{2}(t_{0})-y_{2}\right\|<2+\frac{\varepsilon}{2}.\]
Define a functional $f^*\in S_{C(K,X)^{*}}$ by $f^{*}(f)=x^{*}(f(t_{0}))$ for every $f\in C(K,X)$.
Obviously, $f_{1}\in S_{f^*}:=S(f^*,\varepsilon)$, and there is a continuous map $\phi:K\rightarrow [0,1]$  which satisfies
\begin{center}
$\phi(t_{0})=1$ \quad and \quad $\phi(t)=0$ \quad if $\|f_2(t)-f_{2}(t_{0})\|\geqslant\frac{\varepsilon}{4}$.
\end{center}
Let $g_{i}(t)=\phi(t)y_{i}+(1-\phi(t))f_{2}(t)$ for every $t\in K$ and for $i=1,2$. Then it is easily checked that $g_{1}\in S_{f^*}$ and $g_{2}\in -S_{f^*}$. Moreover,
\begin{equation*}
\|g_1-f_2\|+\|f_2-g_2\|<2+\varepsilon.
\end{equation*}
Hence $C(K,X)$ has the property ($**$).

Now let us prove the ``if" part. For any $x_{1},x_{2}\in S_{X}$, let $f_{1}=x_{1}\chi_{K}$ and $f_{2}=x_{2}\chi_{K}$. Then we have $f_{1},f_{2}\in S_{C(K,X)}$. Since $C(K,X)$ has the property ($**$), for every $\varepsilon>0$ there exists an $f^{*}\in S_{C(K,X)^{*}}$ such that $f_{1}\in S_{f^*}:=S(f^{*},\frac{\varepsilon}{8})$ and
\[\mbox{dist}\left(f_{2},S_{f^*}\right)+\mbox{dist}\left(f_{2},-S_{f^*}\right)<2+\frac{\varepsilon}{8}.\]
 This means that there are $g_{1},-g_{2}\in S_{f^*}$ such that
 \begin{equation*}
 \|f_{2}-g_{1}\|+\|f_{2}-g_{2}\|<2+\frac{\varepsilon}{4}.
\end{equation*}
Note that we can find a $t_{0}\in K$ such that $\|g_{1}-g_{2}+x_{1}\chi_{K}\|=\|g_{1}(t_{0})-g_{2}(t_{0})+x_{1}\|$. By the Hahn-Banach theorem, there exists an $x^{*}\in S_{X^{*}}$ such that
\begin{equation*}
x^{*}(g_{1}(t_{0})-g_{2}(t_{0})+x_{1})=\|g_{1}(t_{0})-g_{2}(t_{0})+x_{1}\|.
\end{equation*}
Set $y_{1}:=g_{1}(t_{0})$, $y_{2}:=g_{2}(t_{0})$ and $S_{x^*}:=S(x^{*},\varepsilon)$. We deduce from
\begin{equation*}
  \|g_{1}(t_{0})-g_{2}(t_{0})+x_{1}\|\geq f^{*}(g_{1}-g_{2}+x_{1}\chi_{K})>3-\frac{\varepsilon}{2}
\end{equation*}
 that $x^{*}(x_{1})>1-\varepsilon/2$. Otherwise,
\begin{align*}
 3-\frac{\varepsilon}{2} &<\|g_{1}(t_{0})-g_{2}(t_{0})+x_{1}\|\\&=x^{*}(g_{1}(t_{0})-g_{2}(t_{0})+x_{1})\leq1+1+1-\frac{\varepsilon}{2}=3-\frac{\varepsilon}{2},
\end{align*}
a contradiction. Thus $x_{1}\in S_{x^*}$. In a similar way, we can obtain $y_{1},-y_{2}\in S_{x^*}$. Moreover, it is easy to see that
\begin{align*}
\|x_{2}-y_{1}\|+\|x_{2}-y_{2}\|&=\|f_{2}(t_0)-g_{1}(t_{0})\|+\|f_{2}(t_0)-g_{2}(t_{0})\|\\
&\leq\|f_{2}-g_{1}\|+\|f_{2}-g_{2}\|<2+\varepsilon.
\end{align*}
So $X$ has the property ($**$). The proof is complete.
 \end{proof}
We will deal with the property ($**$) for $L_\infty(\mu,X)$ and $L_1(\mu,X)$.
Very recently, it has been shown in \cite[Theorem 4.8]{H2} that if $X$ has the property ($**$), then $L_1(\mu, X)$ and $L_\infty(\mu,X)$ also have the property ($**$). In fact, even more general reduction theorem is proved in \cite{H2} for a large class of spaces, such as octahedral and almost square spaces, lush spaces and so on. However, we do not think the converse of the previous result, that is if $L_1(\mu,X)$ or $L_\infty(\mu,X)$ has the property ($**$), then so does $X$, can be deduced from this reduction theorem. Additionally, it may be necessary to provide a direct proof for the fact that $L_1(\mu,X)$ and $L_\infty(\mu,X)$ enjoy the property ($**$) whenever $X$ does.

 To simplify the notation, we will use the following notation during the proof of the theorems:
\begin{equation*}
\Sigma^+:=\{A\in\Sigma: 0<\mu(A)<\infty\}.
\end{equation*}

 \begin{theorem} \label{infty-theorem}
Let $X$ be a Banach space, and let $(\Omega,\Sigma,\mu)$ be a $\sigma$-finite measure space. Then $X$ has the property ($**$) if and only if $L_{\infty}(\mu,X)$ has the property ($**$).
\end{theorem}
\begin{proof}
Suppose first that $X$ has the property ($**$).
 Let $f_{1},f_{2}\in S_{L_{\infty}(\mu,X)}$ and $\varepsilon>0$.  Note that every function in $L_\infty(\mu,X)$ is essentially separably valued. Thus there is an $A_1\in\Sigma^+$ and $x_1\in S_X$ such that
 \begin{equation*}
 \|x_1\chi_{A_1}-f_1\chi_{A_1}\|_\infty<\frac{\varepsilon}{4}.
\end{equation*}
 Consider the function $f_2\chi_{A_1}$. We may also find $A_{2}\in\Sigma^+$ and $x_{2}\in B_X$ such that $A_2\subset A_1$ and $$\|f_{2}\chi_{A_{2}}-x_{2}\chi_{A_{2}}\|_\infty<\frac{\varepsilon}{4}.$$
 Since $X$ has the property ($**$), we can find $x^*\in S_{X^*}$ such that $x_1\in S_{x^*}:=S(x^*,\varepsilon/4)$ and $y_1,-y_2\in S_{x^*}$ satisfying
 \begin{equation*}
   \|y_1-x_2\|+\|x_2-y_2\|<2+\frac{\varepsilon}{4}.
 \end{equation*}
 With $A_2$ and $x^*$ in hand, we can define a functional $f^*\in S_{C(K,X)^*}$ by
 \begin{equation*}
 f^*(f)=x^*\Big(\frac{1}{\mu(A_2)}\int_{A_2}f d\mu\Big)
 \end{equation*}
 for all $f\in C(K,X)$.  Set $g_1:=y_1\chi_{A_2}+f_2\chi_{\Omega\setminus A_2}$ and $g_2:=y_2\chi_{A_2}+f_2\chi_{\Omega\setminus A_2}$.
  Then it is obvious that $f_1, g_1,-g_2\in S_{f^*}:=S(f^*,\varepsilon)$ and
 \begin{align*}
  \|g_1-f_2\|_\infty+\|f_2-g_2\|_\infty&=\|y_1\chi_{A_2}-f_2\chi_{A_2}\|_\infty+\|f_2\chi_{ A_2}-y_2\chi_{A_2}\|_\infty\\
                         &\leq\|y_1-x_2\|+\|x_2-y_2\|+\frac{1}{2}\varepsilon<2+\varepsilon.
 \end{align*}
This thus proves that $L_\infty(\mu,X)$ has the property ($**$).

 Now we deal with the converse.
 Fix $x_{1},x_{2}\in S_{X}$ and $A\in\Sigma^+$. Set $f_{1}=x_{1}\chi_{A}$ and $f_{2}=x_{2}\chi_{A}$. That $L_{\infty}(\mu,X)$ has the property ($**$) produces $f^{*}\in S_{L_{\infty}(\mu,X)^{*}}$ such that $f_{1}\in S_{f^*}:=(f^{*},\frac{\varepsilon}{8})$ and $g_1,-g_2\in S_{f^*}$ such that
 \begin{equation*}
\|g_1-f_{2}\|_\infty+\|f_{2}-g_{2}\|_\infty<2+\frac{\varepsilon}{4}.
 \end{equation*}
Observe that $\|g_1-g_2+f_1\|\geq f^*(g_1-g_2+f_1)>3-\varepsilon/2$. Therefore, there exists $B\subset A$ with $B\in\Sigma^+$ such that
\begin{equation*}
\|g_{1}(t)-g_{2}(t)+x_1\|>3-\frac{\varepsilon}{2}.
\end{equation*}
for all $t\in B$. Similar arguments as above show that there are $y_1,y_2\in B_X$ and $C\in \Sigma^+$ such that $C\subset B$ and
\begin{align*}
  \|y_1\chi_C-g_1\chi_C\|_\infty<\frac{\varepsilon}{8}  \,\, \mbox{and} \,\, \|y_2\chi_C-g_2\chi_C\|_\infty<\frac{\varepsilon}{8}.
\end{align*}
It follows that
\begin{equation*}
  \|y_1-y_2+x_1\|>3-\varepsilon.
\end{equation*}
The Hahn-Banach theorem ensures us that there is a functional $x^*\in S_X$ such that
\begin{equation*}
  x^*(y_1-y_2+x_1)>3-\varepsilon.
\end{equation*}
It follows that $y_1,-y_2,x_1\in S(x^*,\varepsilon)$, and more,
\begin{equation*}
  \|y_1-x_2\|+\|x_2-y_2\|<2+\varepsilon.
\end{equation*}
Thus $X$ has the property ($**$).
\end{proof}
In fact, a minor modification of the proof of \cite[Propsition 2.2]{THL} can provide a stronger conclusion.  This conclusion yields the equivalence of generalised-lushness and the property ($**$) for separable spaces which was previously noted in \cite{H}. We also apply it to show that $X$ has the property ($**$) whenever $L_1(\mu,X)$ does.
Thus for our particular use, we include here its proof.

Given a Banach
space $X$, a subset $G \subset X^*$ is called norming if
$\|x\|=\sup\{ |x^*(x)| : x^*\in G \}$
for every $x\in X$.
\begin{proposition}\label{propsition:1}
 Let $X$ be a Banach space having the property ($**$), and let $X_0\subset X$ be a separable subspace. Suppose that $G\subset S_{X^*}$ is norming and symmetric.
 Then for every $\varepsilon>0$, the set
\begin{align*}
\{x^*\in G: \mbox{dist}(y, S)+\mbox{dist}(y,-S)<2+\varepsilon \
\mbox{ for all}\  y\in S_{X_0} , \mbox{ where } S=S(x^*, \varepsilon )\}
\end{align*}
is a weak$^*$ $G_\delta$-dense subset of the weak$^*$ closure of $G$. In particular, if $X$ is separable, then $X$ is a GL-space.
\end{proposition}
\begin{proof}
Let $\{y_n\}\subset S_{X_0}$ be a sequence dense in $S_{X_0}$.  Fix $0<\varepsilon<1$.  Given $n\geq1$, set
\begin{align*}
K_n=\{x^*\in G:  \ \ \mbox{dist}(y_n, S)+\mbox{dist}(y_n,-S)<2+\varepsilon\ \ \mbox{where} \  \ S=S(x^*, \varepsilon )\}.
\end{align*}
Then $K_n$ is weak$^*$-open and $\overline{K_n}^{\omega^*}=\overline{G}^{\omega^*}$.
Indeed, if $x^*\in K_n$, there exist $x_n\in S(x^*,\varepsilon)$ and $z_n\in -S(x^*,\varepsilon)$ such that
\begin{align*}
\|x_n-y_n\|+\|y_n-z_n\|<2+\varepsilon.
\end{align*}
Let $$U=\{y^*\in G: y^*(x_n)>1-\varepsilon \  \ \mbox {and}  \ \ y^*(-z_n)>1-\varepsilon\} .$$ Then it is easily checked that $U$ is a weak$^*$-neighborhood of $x^*$ in $G$ satisfying  $U\subset K_n$. Thus $K_n$ is weak$^*$-open.

 To prove $\overline{K_n}^{\omega^*}=\overline{G}^{\omega^*}$, it is enough to show that $G\subset \overline{K_n}^{\omega^*}$. Since \cite [Lemma 3.40]{F} states that for every $x^*\in G$, the weak$^*$-slices containing $x^*$ form a neighborhood base of $x^*$, it suffices to prove that for every $x\in S$,  the weak$^*$-slice $S(x,\varepsilon_1)\cap K_n\neq\emptyset$ for all $\varepsilon_1\in (0, \varepsilon)$.  Since $X$ has the property ($**$), there is a slice $S_{y^*}:=S(y^*,\varepsilon_1/3)$ with $y^*\in S_{X^*}$ such that
\begin{align*}
x\in S_y^*\    \ \mbox{and}\   \ \mbox{dist}(y_n,S_{y^*})+\mbox{dist}(y_n,-S_{y^*})<2+\varepsilon_1.
\end{align*}
Thus we may find $x_n'\in S_{y^*}$ and $z_n' \in -S_{y^*}$ such that
\begin{align*}
\|x_n'-y_n\|+\|y_n-z_n'\|<2+\varepsilon_1 \  \ \mbox{and}\ \ \|x+x_n'-z_n'\|>3-\varepsilon_1.
\end{align*}
Note that $G$ is norming and symmetric. Thus there is a $z^*\in  G$ such that
\begin{align*}
z^*(x+x_n'-z_n')>3-\varepsilon_1.
\end{align*}
This implies that $z^*\in S(x, \varepsilon_1)\cap K_n$.

Now set $K=\bigcap_{n\in\mathbb{N}} K_n$. Then by the Baire theorem, $K$ is a weak$^* $ $G_\delta$-dense subset of $\overline{G}^{\omega^*}$. This together with density of $(y_n)$ in $S_{X_0}$ gives the first conclusion and the second conclusion is clear.
\end{proof}
Let us make a remark here. Proposition \ref{propsition:1} combined with Theorem \ref{CK-theorem} establishes that if $C(K,X)$ is a GL-space, then so is $X$ under the assumption that $X$ is separable. The same result holds for the space $L_\infty(\mu,X)$. We do not know if this is true in the general.
 Throughout what follows, we will use the notation
\begin{equation*}
\mathcal{S}(x^*,\alpha):=\{x\in X: x^*(x)>\|x\|-\alpha\},
\end{equation*}
where $x^*\in S_{X^*}$ and $0<\alpha<1$. In this notation, it is obvious that $\mathcal{S}(x^*,\alpha)$ contains the general slice $S(x^*,2\alpha)$ for $0<\alpha<\frac{1}{2}$.

 To show that $X$ has the property ($**$) whenever $L_1(\mu,X)$ does, we need some more lemmas.
\begin{lemma}\label{lem:2}
Let $X$ be a Banach space, and let $y$ be in $S_X$. For every $0<\varepsilon<1$, if there are $x^*\in S_{X^*}$, $x_1\in \mathcal{S}(x^*,\varepsilon/3)$, $x_2\in -\mathcal{S}(x^*,\varepsilon/3)$ such that
\begin{equation*}
  \|x_1-y\|+\|y-x_2\|<\|x_1\|+\|x_2\|+\frac{\varepsilon}{3},
\end{equation*}
then we have $x_1-y, y-x_2\in \mathcal{S}(x^*, \varepsilon)$.
\end{lemma}
\begin{proof}
The proof of the two cases $x_1-y\in \mathcal{S}(x^*, \varepsilon)$  and $y-x_2\in \mathcal{S}(x^*, \varepsilon)$  are completely the same. It is enough to prove the first one. Assume, on the contrary, that
$x^*(x_1-y)\leq \|x_1-y\|-\varepsilon$. Then
\begin{align*}
  \|x_1-y\|+\|y-x_2\|&\geq x^*(x_1-y)+\varepsilon+x^*(y-x_2)\\&=x^*(x_1-x_2)+\varepsilon>\|x_1\|+\|x_2\|+\frac{\varepsilon}{3}.
\end{align*}
A contradiction therefore completes the proof.
\end{proof}
\begin{remark}
One can easily check that a converse version of the previous lemma remains true. To be precise, if $x_1-y, y-x_2\in \mathcal{S}(x^*, \varepsilon)$, then \begin{equation*}
  \|x_1-y\|+\|y-x_2\|\leq x^*(x_1-y)+x^*(y-x_2)+2\varepsilon\leq \|x_1\|+\|x_2\|+2\varepsilon.
\end{equation*}
This observation actually provides an approach to find a slice which satisfies \eqref{equ:24}.
\end{remark}
A simple but very useful numerical result appears in \cite[Lemma 8.13]{MRK}. We will also apply it to deal with the property ($**$) in the space $L_1(\mu,X)$. We give the proof for the sake of completeness.
\begin{lemma}\label{lem:1}
Let $\varepsilon>0$ $\delta>0$, and let $\lambda_i\geq 0$ for all $i=1,\cdots,n$. Suppose that $\alpha_i, \beta_i\in\mathbb{R}$ are such that $\alpha_i\leq\beta_i$ for all $i=1,\cdots,n$ and satisfy $(\sum_{i=1}^n\lambda_i\beta_i)-\varepsilon \delta<\sum_{i=1}^n\lambda_i\alpha_i$. Then
\begin{equation*}
  \sum\{\lambda_i: \beta_i-\alpha_i\geq\varepsilon\}<\delta.
\end{equation*}
In particular, if $\sum_{i=1}^n\lambda_i=1$, then
\begin{equation*}
  \sum\{\lambda_i: \beta_i-\alpha_i<\varepsilon\}>1-\delta.
\end{equation*}
\end{lemma}
\begin{proof}
Set $I=\{1\leq i\leq n:\beta_i-\alpha_i\geq\varepsilon\}$. Then it is easily seen that
\begin{align*}
  \sum_{i=1}^{n}\lambda_i\beta_i=\sum_{i\in I} \lambda_i\beta_i+\sum_{i\notin I} \lambda_i\beta_i&\geq\sum_{i\in I} \lambda_i(\alpha_i+\varepsilon)+\sum_{i\notin I} \lambda_i\alpha_i\\&=\sum_{i=1}^{n}\lambda_i\alpha_i+\varepsilon\sum_{i\in I} \lambda_i.
\end{align*}
It follows immediately from this and the hypothesis that $\sum_{i\in I} \lambda_i<\delta$. The second conclusion is obvious.
\end{proof}

The same results as \cite[Theorem 2.11]{THL}
also hold for the property ($**$). Although the proofs are actually analogous to those of \cite[Theorem 2.11]{THL}, we give the proof of the $l_1$-sum case since this result is necessary in what follows.
\begin{proposition}\label{proposition2}
Let $\{E_\lambda : \lambda\in\Lambda\}$ be a family of Banach spaces, and let
$E = [\bigoplus_{\lambda\in\Lambda} E_\lambda]_{F}$ where $F=c_0, \, l_\infty \ \mbox{or } \,  l_1$.
Then $E$  has the property ($**$) if and only if each $E_\lambda$ has the property ($**$).
\end{proposition}
\begin{proof}
In the $l_1$-sum case, we first show the ``if" part. Given $x=(x_\lambda), y=(y_\lambda)\in S_E$ and $\varepsilon>0$, for each $\lambda$ with $x_\lambda\neq 0$, there is a corresponding slice  $S_\lambda:=S(x_\lambda^*, \varepsilon)$ with $x_\lambda^*\in S_{E_\lambda^*}$ such that
\begin{align*}
x_\lambda^*(x_\lambda)>(1-\varepsilon )\|x_\lambda\| \  \mbox{and} \   \ \mbox{dist}(\frac{y_\lambda}{\|y_\lambda\|}, S_\lambda)+\mbox{dist}(\frac{y_\lambda}{\|y_\lambda\|},-S_\lambda)<2+\varepsilon,
\end{align*}
where $y_\lambda\neq0$. Then $x^*=(x_\lambda^*)\in S_{E^*}$ with $x_\lambda^*=0$ whenever $x_\lambda=0$, and the required slice satisfying (\ref{equ:24}) is $S(x^*,\varepsilon)$. Therefore $E$ has the property ($**$).

For the ``only if" part, fix $x_\lambda,y_\lambda \in S_{E_\lambda}$ and $0<\varepsilon<1/16$. Then $x=(x_\delta),y=(y_\delta)\in S_E$ where $x_\delta=y_\delta=0$ for all $\delta\neq \lambda$. Since $E$ has the property ($**$), there is an $x^*=(x_\delta^*)\in S_{E^*}$ with $S:=S(x^*,\varepsilon^2/4)$ such that
$$ x\in S\ \  \mbox{and} \   \ \mbox{dist}(y, S)+\mbox{dist}(y,-S)<2+\frac{\varepsilon^2}{4}.$$  We will prove that the slice $S_\lambda:=S(x_\lambda^*/\|x_\lambda^*\| , \varepsilon)$ is the desired one.

It is easily checked that $x_\lambda\in S_\lambda$ and there are $u=(u_\delta)\in S$ and $v=(v_\delta)\in -S$ such that
\begin{align}\label{equ:30}
\|y-u\|+\|y-v\|<2+\frac{\varepsilon^2}{4}.
\end{align}
It follows from the definition of $E$ that
\begin{align}
\|y-u\|+\|y-v\|&=\|y_\lambda-u_\lambda\|+\sum_{\delta\neq \lambda}\|u_\delta\|+\|y_\lambda-v_\lambda\|+\sum_{\delta\neq \lambda}\|v_\delta\| \nonumber\\
&>\|y_\lambda-u_\lambda\|+1-\varepsilon^2/4-\|u_\lambda\|+\|y_\lambda-v_\lambda\|+1-\varepsilon^2/4-\|v_\lambda\|\nonumber \\
&=\|y_\lambda-u_\lambda\|-\|u_\lambda\|+\|y_\lambda-v_\lambda\|-\|v_\lambda\|+2-\varepsilon^2/2.\label{equ:31}
\end{align}
We deduce from (\ref{equ:30}) and (\ref{equ:31}) that
\begin{align}\label{equ:33}
\|y_\lambda-u_\lambda\|+\|y_\lambda-v_\lambda\|<\|u_\lambda\|+\|v_\lambda\|+\varepsilon^2.
\end{align}
On the other hand,
\begin{align}\label{equ:34}
x_\lambda^*(u_\lambda)>1-\varepsilon^2/4-\sum_{\delta\neq \lambda}\|u_\delta\|\geq1-\varepsilon^2/4-1+\|u_\lambda\|=\|u_\lambda\|-\varepsilon^2/4,
\end{align}
and similarly,
\begin{align}\label{equ:35}
x_\lambda^*(-v_\lambda)>\|v_\lambda\|-\frac{\varepsilon^2}{4}.
\end{align}
We apply \eqref{equ:33}, \eqref{equ:34}, \eqref{equ:35} and Lemma \ref{lem:2} to get that
\begin{equation}\label{equ:36}
x_\lambda^*(u_\lambda-y_\lambda)\geq \|u_\lambda-y_\lambda\|-3\varepsilon^2
\end{equation}
and
\begin{equation*}
x_\lambda^*(y_\lambda-v_\lambda)\geq \|y_\lambda-v_\lambda\|-3\varepsilon^2.
\end{equation*}

Therefore, if $\|u_\lambda\|\leq \varepsilon/4$, \eqref{equ:36} yields
\begin{equation*}
x^*_\lambda(-y_\lambda)>1-\varepsilon/4-3\varepsilon^2-\varepsilon/4>1-\varepsilon.
\end{equation*}
This means that $-y_\lambda\in S_\lambda$. Clearly it satisfies $$\mbox{dist}(y_\lambda, S_\lambda)+\mbox{dist}(y_\lambda, -S_\lambda)\leq2<2+\varepsilon.$$
A similar result holds in the case of $\|v_\lambda\|\leq \varepsilon/4$. It remains to consider the case that $\|v_\lambda\|>\varepsilon/4$ and $\|v_\lambda\|>\varepsilon/4$. Put $w_\lambda:=u_\lambda/\|u_\lambda\|$ and $t_\lambda:=v_\lambda/\|v_\lambda\|$. Then $w_\lambda, -t_\lambda \in S_\lambda$ following from  \eqref{equ:34} and \eqref{equ:35} respectively. The desired estimate
\begin{align*}
\|y_\lambda-w_\lambda\|+\|y_\lambda-t_\lambda\|<2+\varepsilon
\end{align*}
 is got directly from (\ref{equ:33}). The proof is complete.
\end{proof}
Now we are ready to work with the property ($**$) for the space $L_1(\mu, X)$.

 \begin{theorem}
Let $X$ be a Banach space, and let $(\Omega,\Sigma,\mu)$ be a $\sigma$-finite measure space. Then $X$ has the property ($**$) if and only if $L_{1}(\mu,X)$ has the property ($**$).
\end{theorem}
\begin{proof}
Since $L_1(\mu,X)$ is isometrically isomorphic to an $l_1$-sum of spaces $L_1(\mu_i,X)$ for some finite measures $\mu_i$, we deduce from Proposition \ref{proposition2} that it is enough to deal with finite measure, and by normalizing
the measure, we may assume that $\mu(\Omega)=1$.

Assume that $X$ has the property ($**$). To prove that so does $L_1(\mu, X)$, we will check that \eqref{equ:24} is satisfied. Given $f,g\in S_{L_1(\mu,X)}$ and $\varepsilon>0$, we apply \cite[Lemma III.2.1]{Die} to obtain a partition $\pi$ of $\Omega$ into a finite family of disjoint members of $\Sigma^+$ such that
\begin{align}\label{equ:60}
\|E_{\pi}(f)-f\|_1<\frac{\varepsilon}{8}.
\end{align}
and
\begin{align*}
 \|E_{\pi}(g)-g\|_1<\frac{\varepsilon}{8}.
\end{align*}
 where $E_{\pi}:L_1(\mu,X)\rightarrow L_1(\mu,X)$ is a contractive projection given by
\[E_{\pi}(h)=\sum_{A\in\pi}(\frac{1}{\mu(A)}\int_Ah\,d\mu)\chi_A,\] for all $h\in L_1(\mu,X).$

We set $x_A:=\int_A f \,d\mu$ and $y_A:=\int_A g \,d\mu$. Since $X$ has the property ($**$), there exists an $x_A^*\in S_{X^*}$ with $S_{x_A^*}=S(x_A^*,\varepsilon)$ such that
\begin{align}\label{x_A}
x_A^*(x_A)\geq (1-\frac{\varepsilon}{2})\|x_A\|
\end{align}
and
\begin{align*}
 \|y_A-\|y_A\|z_A^+\|+\|y_A-\|y_A\|z_A^-\|\leq(2+\frac{\varepsilon}{2})\|y_A\|
\end{align*}
where $z_A^+,-z_A^-\in S_{x_A^*}$. Now we can define a functional $f^*\in L_1(\mu,X)^*$ by
$$f^*(h)=\sum_{A\in\pi}x_A^*(\int_A h\,d\mu)$$
for all $h\in L_1(\mu,X).$ Then clearly $f^*\in S_{L_1(\mu,X)^*}.$
We will check that the slice $S_{f^*}=S(f^*,\varepsilon)$ is the desired one. Observe that $f\in S_{f^*}$ is an immediate consequence of (\ref{x_A}) and \eqref{equ:60}.

 Consider the functions $h^+, h^-\in L_1(\mu,X)$ defined by
\[h^+=\sum_{A\in\pi}(\frac{\|y_A\|}{\mu(A)}z^+_A)\chi_A \ \ \mbox{and}\ \ h^-=\sum_{A\in\pi}(\frac{\|y_A\|}{\mu(A)}z^-_A)\chi_A.\]
By the definition of $f^*$ and the partition $\pi$, we see that
$$h^{+}\in S_{f^*} \,\ \mbox{and}\ \, h^{-}\in -S_{f^*}.$$
Furthermore, an easy computation shows that
\begin{align*}
&\|g-h^+\|_1+\|g-h^-\|_1\\
\leq&\|E_{\pi}(g)-g\|_1+\|E_{\pi}(g)-h^+\|_1+\|E_{\pi}(g)-g\|_1+\|E_{\pi}(g)-h^-\|_1\\
\leq&2+\varepsilon/2+\varepsilon/4<2+\varepsilon.
\end{align*}
This thus proves that $L_1(\mu,X)$ has the property ($**$).

 For the converse, we will draw an idea from \cite[Theorem 8.10.(b)]{MRK} where Lemma \ref{lem:1} is applied.
 Fix $x,y\in S_X$, and for every $0<\varepsilon<1/4$, choose $\eta\in(0,1)$ such that $\eta<(\varepsilon/4)^6$. It suffices to show that there is an $x^*\in S_{X^*}$ such that \eqref{equ:24} holds. The hypothesis provides a $g^*\in S_{{L_{1}(\mu,X)}^*}$ such that $x\chi_{\Omega}\in S_{g^*}:=S(g^*, \eta^9/3)$ and
\begin{equation*}
  \mbox{dist}(y\chi_{\Omega}, S_{g^*})+\mbox{dist}(y\chi_{\Omega}, -S_{g^*})<2+\frac{\eta^9}{3}.
\end{equation*}
This and the density of the simple functions in $L_{1}(\mu,X)$ imply that there exist simple functions $g_1\in S_{g^*}$ and $g_2\in -S_{g^*}$ such that
\begin{equation}\label{equ:2}
  \|y\chi_{\Omega}-g_1\|_1+\|y\chi_{\Omega}-g_2\|_1<2+\frac{\eta^9}{3}.
\end{equation}
We may write $g_1=\Sigma_{i=1}^{n}x_i\chi_{A_i}\in S_{g^*}$ and $g_2=\Sigma_{i=1}^{n}y_i\chi_{A_i}\in-S_{g^*}$, where $x_i,y_i\in X$ and $\{A_i\}_{i=1}^{n}\subset \Sigma^+$ is a finite partition of $\Omega$. For each $i=1,\cdots,n$, define a functional $y^*_i\in X^*$ by
 \begin{equation*}
 y_i^*(z)=g^*\Big(\frac{z\chi_{A_i}}{\mu(A_i)}\Big)   \quad ( z\in X).
\end{equation*}
Then it is clear that $\|y_i^*\|\leq 1$ for $i=1,\cdots,n$, and
\begin{align}
  &\sum_{i=1}^{n}y_i^*(x)\mu(A_i)=g^*(x\chi_{\Omega})>1-\frac{\eta^9}{3},\label{equ:44}  \\
   &\sum_{i=1}^{n}y_i^*(x_i)\mu(A_i)=g^*(g_1)>1-\frac{\eta^9}{3}\label{equ:45}
 \end{align}
and
\begin{equation}\label{equ:46}
\sum_{i=1}^{n}y_i^*(-y_i)\mu(A_i)=g^*(-g_2)>1-\frac{\eta^9}{3}.
\end{equation}
Furthermore, by \eqref{equ:2} and Lemma \ref{lem:2}, we have
\begin{equation*}
  f^*(g_1-y\chi_\Omega)>\|g_1-y\chi_\Omega\|_1-\eta^9\,\, \mbox{and}\,\, f^*(y\chi_\Omega-g_2)>\|g_1-y\chi_\Omega\|_1-\eta^9.
\end{equation*}
That is
\begin{equation}\label{equ:47}
  \sum_{i=1}^{n} y_i^*(x_i-y)\mu(A_i)>\sum_{i=1}^{n}\|x_i-y\|\mu(A_i)-\eta^9
\end{equation}
and
\begin{equation}\label{equ:48}
  \sum_{i=1}^{n} y_i^*(y-y_i)\mu(A_i)>\sum_{i=1}^{n}\|y-y_i\|\mu(A_i)-\eta^9.
\end{equation}

Observe that $y^*_i(z)\leq\|z\|$ for all $z\in X$ and for each $i=1,\cdots,n$. Then applying Lemma \ref{lem:1} to the above inequalities \eqref{equ:44}-\eqref{equ:48}, we clearly get
\begin{equation}\label{equ:38}
 \sum\{\mu(A_i): z_i\in \mathcal{S}(y_i^*, \eta^3)\}>1-\eta^6,
\end{equation}

for $\{z_i\}_{i=1}^{n}\in\Big\{\{x\}_{i=1}^n,\{x_i\}_{i=1}^n,\{-y_i\}_{i=1}^{n},\{x_i-y\}_{i=1}^{n},\{y-y_i\}_{i=1}^{n}\Big\}$.

On the other hand, note that
\begin{equation*}
\sum_{i=1}^{n}\|x_i\|\mu(A_i)=\|g_1\|=1.
\end{equation*}
Thus
\begin{equation*}
\sum\{\mu(A_i):\|x_i\|>1+ \eta^3\}<\frac{1}{1+\eta^3}.
\end{equation*}
So
\begin{equation}\label{equ:11}
\sum\{\mu(A_i):\|x_i\|\leq 1+ \eta^3\}>\frac{\eta^3}{1+\eta^3}.
\end{equation}
Since $\eta<\varepsilon^3/64<1/64$, we deduce from \eqref{equ:38} and \eqref{equ:11} that there is some $0\leq i_0\leq n$ such that
\begin{equation}\label{equ:12}
\|x_{i_0}\|\leq 1+ \eta^3
\end{equation}
and
\begin{equation}\label{equ:39}
\{x,x_{i_0},-y_{i_0},x_{i_0}-y,y-y_{i_0}\}\subset S(y_{i_0}^*,\eta^3).
\end{equation}

 For our conclusion, the argument will be divided into three cases.

If $\|x_{i_0}\|\leq \eta$, using that $\|y-x_{i_0}\|\geq 1-\eta$, we apply \eqref{equ:39} to conclude that
\begin{align*}
   y^*_{i_0}(-y)&\geq\|x_{i_0}-y\|-\eta^3-y^*_{i_0}(x_{i_0})\\
  &\geq 1-2\eta-\eta^3>1-\varepsilon.
\end{align*}
Then $x,-y\in S_{y_{i_0}^*}:=S(y_{i_0}^*,\varepsilon)$, and thus $\mbox{dist}(y,-S_{y_{i_0}^*})=0$. So \eqref{equ:24} is already verified.

A similar proof shows that if $\|y_{i_0}\|\leq \eta$, then
\begin{align*}
   y^*_{i_0}(y)\geq 1-2\eta-\eta^3>1-\varepsilon.
\end{align*}
It follows that $S_{y_{i_0}^*}$ is just the desired slice.

 The previous argument also implies that it is only possible that $\|y_{i_0}\|\leq \eta$ or $\|x_{i_0}\|\leq \eta$ since $\|y_{i_0}\|\leq \eta$ and $\|x_{i_0}\|\leq \eta$ cannot hold simultaneously. Thus the remaining case that needs to deal with is that $\|y_{i_0}\|>\eta$ and $\|x_{i_0}\|>\eta$. This and \eqref{equ:39} guarantee that
\begin{equation}\label{equ:50}
y^*_{i_0}\Big(\frac{x_{i_0}}{\|x_{i_0}\|}\Big)>1-\eta^2
\end{equation}
and
\begin{equation*}
 y^*_{i_0}\Big(-\frac{y_{i_0}}{\|y_{i_0}\|}\Big)>1-\eta^2.
\end{equation*}
Moreover, \eqref{equ:39} combined with \eqref{equ:12} establishes that
\begin{align}
y^*_{i_0}\Big(\frac{x_{i_0}}{\|x_{i_0}\|}-y\Big)&\geq\|x_{i_0}-y\|-\eta^3-y^*_{i_0}\Big(x_{i_0}-\frac{x_{i_0}}{\|x_{i_0}\|}\Big)\nonumber\\
&\geq\|x_{i_0}-y\|-\eta^3-(\|x_{i_0}\|-1+\eta^2) \nonumber \\
&\geq\Big\|\frac{x_{i_0}}{\|x_{i_0}\|}-y\Big\|-\Big|1-\|x_{i_0}\|\Big|+1-\|x_{i_0}\|-\eta^3-\eta^2 \nonumber \\
&\geq\Big\|\frac{x_{i_0}}{\|x_{i_0}\|}-y\Big\|-2\eta^3-\eta^3-\eta^2 \nonumber \\
&>\Big\|\frac{x_{i_0}}{\|x_{i_0}\|}-y\Big\|-\varepsilon^9. \label{equ:15}
\end{align}
In fact, the proof will be done provided that \eqref{equ:12} also holds for $y_{i_0}$. However, this cannot be obtained directly. For this reason, we still need to consider the vector $y\chi_{A_{i_0}}/\mu(A_{i_0})\in S_{L_1(\mu,X)}$. Note that for each finite partition $\{A_1,\cdots, A_n\}$ of $\Omega$ and finite vectors $\{x_1\cdots,x_n\}\subset S_X$,  $X_0=\mbox{span}\{x_i\chi_{A_i}: 1\leq i\leq n\}$ is an $n$-dimensional Banach space. By Proposition \ref{propsition:1},
 we may assume that there are simples $f^{+}, -f^{-} \in S(g^*,\eta^9/3)$ such that
\begin{equation*}
  \|f^+-\frac{y\chi_{A_{i_0}}}{\mu(A_{i_0})}\|_1+\|\frac{y\chi_{A_{i_0}}}{\mu(A_{i_0})}-f^{-}\|_1<2+\frac{\eta^9}{3}.
\end{equation*}
We may write
\begin{equation*}
f^{+}=\Sigma_{j=1}^m x_{i_0,j}^{+}\chi_{A_{i_0,j}}+\Sigma_{j=m+1}^{k}x^+_{j}\chi_{B_j}
\end{equation*}
 and
\begin{equation*}
f^{-}=\Sigma_{j=1}^m x_{i_0,j}^{-}\chi_{A_{i_0,j}}+\Sigma_{j=m+1}^{k}x^-_{j}\chi_{B_j},
\end{equation*}
where $\{A_{i_0,1},\cdots, A_{i_0,m}, B_{m+1},\cdots,B_k\}\subset\Sigma^+$ is a finite partition of $\Omega$ such that
$\cup_{j=1}^{m} A_{i_0,j}=A_{i_0}$. Similarly as above, define $y_{i_0,j}^*, x^*_j\in B_{X^*}$ respectively by
\begin{equation*}
y_{i_0,j}^*(z)=g^*\Big(\frac{z\chi_{A_{i_0,j}}}{\mu(A_{i_0,j})}\Big)\,  \quad (z\in X, j=1,\cdots,m)
 \end{equation*}
and
\begin{equation*}
 x_{j}^*(z)=g^*\Big(\frac{z\chi_{B_{j}}}{\mu(B_{j})}\Big) \,\quad  (z\in X, j=m+1,\cdots,k).
\end{equation*}

Since $g^*(f^+)>1-\eta^9$, this together with an observation that
\begin{equation*}
\sum_{j=m+1}^{k}x_j^*(x_{i_0,j}^+)\mu(B_j)\leq\sum_{j=m+1}^{k}\|x_{i_0,j}^+\|\mu(B_j)
\end{equation*}
yields
\begin{align}\label{equ:17}
\sum_{j=1}^{m}y_{i_0,j}^*\big(x_{i_0,j}^+\mu(A_{i_0})\big)\frac{\mu(A_{i_0,j})}{\mu(A_{i_0})}
=&\sum_{j=1}^{m}y_{i_0,j}^*(x_{i_0,j}^+)\mu(A_{i_0,j})\nonumber \\ >&\sum_{j=1}^{m}\|x_{i_0,j}^+\|\mu(A_{i_0,j})-\eta^9\nonumber\\
>&\sum_{j=1}^{m}\big\| x_{i_0,j}^+\mu(A_{i_0})\big\|\frac{\mu(A_{i_0,j})}{\mu(A_{i_0})}-\eta^3 \nonumber\\ >&\sum_{j=1}^{m}\big\|x_{i_0,j}^+\mu(A_{i_0})\big\|\frac{\mu(A_{i_0,j})}{\mu(A_{i_0})}-\varepsilon^9.
\end{align}
Following in the similar line as above, we conclude that \eqref{equ:17} also holds for $\{-x_{i_0,j}^-\}_{j=1}^{m}$, $\{x_{i_0,j}^+-y/\mu(A_{i_0,j})\}_{j=1}^{m}$ and $\{y/\mu(A_{i_0,j})-x_{i_0,j}^-\}_{j=1}^{m}.$
Note that for every $z\in X$, we have
\begin{equation*}
y_{i_0}^*(z)=\sum_{j=1}^{m}y_{i_0,j}^*(z)\frac{\mu(A_{i_0,j})}{\mu(A_{i_0})}.
\end{equation*}
Combining this with \eqref{equ:39}, \eqref{equ:50} and \eqref{equ:15} and noting $\eta<\varepsilon^6$, we obtain
\begin{equation*}
  \sum_{j=1}^{m}y_{i_0,j}^*(z)\frac{\mu(A_{i_0,j})}{\mu(A_{i_0})}>\|z\|-\varepsilon^9
\end{equation*}
for all $z\in \{x,\frac{x_{i_0}}{\|x_{i_0}\|}, \frac{x_{i_0}}{\|x_{i_0}\|}-y\}$.
Thus an application of Lemma \ref{lem:1} again guarantees that
\begin{equation}\label{equ:40}
\sum\{\frac{\mu(A_{i_0,j})}{\mu(A_{i_0})}: z_j\in \mathcal{S}(y^*_{i_0,j}, \varepsilon^3)\}>1-\varepsilon^6,
\end{equation}
for
\begin{equation*}
\{z_j\}\in\Big\{\{x\}, \{\frac{x_{i_0}}{\|x_{i_0}\|}\}, \{\frac{x_{i_0}}{\|x_{i_0}\|}-y\}\Big\}
\end{equation*}
and
\begin{equation*}
\{z_j\}\in\Big\{\{x_{i_0,j}^+\mu(A_{i_0})\}, \{-x_{i_0,j}^-\mu(A_{i_0})\}, \{x_{i_0,j}^+\mu(A_{i_0})-y\},\{y-x_{i_0,j}^-\mu(A_{i_0})\}\Big\}.
\end{equation*}
(Here, we omit the superscript and subscript when confusion
is unlikely).

Observe from $\sum_{j=1}^{m}\|x_{i_0,j}\|\mu({A_{i_0,j}})\leq1$ that
\begin{equation*}
  \sum\{\frac{\mu(A_{i_0,j})}{\mu(A_{i_0})}: \|x_{i_0,j}\|>\frac{1+\varepsilon^3}{\mu(A_{i_0})}\}<\frac{1}{1+\varepsilon^3}.
\end{equation*}
Consequently,
\begin{equation*}
  \sum\{\frac{\mu(A_{i_0,j})}{\mu(A_{i_0})}: \|x_{i_0,j}\|\leq\frac{1+\varepsilon^3}{\mu(A_{i_0})}\}>\frac{\varepsilon^3}{1+\varepsilon^3}.
\end{equation*}
This together with \eqref{equ:40} allows us to conclude that there is a $j_0\in\{1,\cdots,m\}$ such that
\begin{equation}\label{equ:43}
 \|x_{i_0,j_0}^-\mu(A_{i_0})\|\leq1+\varepsilon^3
\end{equation}
and
\begin{equation}\label{equ:41}
\{x, \frac{x_{i_0}}{\|x_{i_0}\|}, \frac{x_{i_0}}{\|x_{i_0}\|}-y,  z_{i_0,j_0}^+, -z_{i_0,j_0}^-, z_{i_0,j_0}^+-y, y-z_{i_0,j_0}^-\Big\}\subset \mathcal{S}(y_{i_0,j_0}^*,\varepsilon^3),
\end{equation}
where $z_{i_0,j_0}^+=x_{i_0,j_0}^+\mu(A_{i_0})$ and $z_{i_0,j_0}^-=x_{i_0,j_0}^-\mu(A_{i_0})$.
Following in an exactly similar way as in the case where $y\chi_\Omega$ is considered, we have
\begin{equation*}
  y_{i_0,j_0}^*(-y)>1-\varepsilon/2-\varepsilon^3>1-\varepsilon \quad \mbox{or} \quad   y_{i_0,j_0}^*(y)>1-\varepsilon/2-\varepsilon^3>1-\varepsilon
\end{equation*}
under the condition that $\|z_{i_0,j_0}^+\|\leq \varepsilon/4$ or $\|z_{i_0,j_0}^-\|\leq\varepsilon/4$, respectively.
We only need to settle the case where $\|z_{i_0,j_0}^+\|>\varepsilon/4$ and $ \|z_{i_0,j_0}^+\|>\varepsilon/4$. A similar argument to that in the case where we get \eqref{equ:15} by using \eqref{equ:43} shows that
\begin{equation*}
 y_{i_0,j_0}^*\Big(y-\frac{x_{i_0,j_0}^-}{\|x_{i_0,j_0}^-\|}\Big)>\Big\|y-\frac{x_{i_0,j_0}^-}{\|x_{i_0,j_0}^-\|}\Big\|-3\varepsilon^3-4\varepsilon^2.
\end{equation*}
On combining this with \eqref{equ:41}, we deduce that
\begin{align*}
&\Big\|\frac{x_{i_0}}{\|x_{i_0}\|}-y\Big\|+\Big\|y-\frac{x_{i_0,j_0}^-}{\|x_{i_0,j_0}^-\|}\Big\|\\
<& y_{i_0,j_0}^*\Big(\frac{x_{i_0}}{\|x_{i_0}\|}-y\Big)+y_{i_0,j_0}^*\Big(y-\frac{x_{i_0,j_0}^-}{\|x_{i_0,j_0}^-\|}\Big)+4(\varepsilon^2+\varepsilon^3)<2+\varepsilon.
\end{align*}
Finally, \eqref{equ:41} proves that the required slice is right $S(y^*_{i_0,j_0},\varepsilon)$.
This completes the proof.
\end{proof}
Let us stress a question on vector-valued function spaces for GL-spaces. It is only known that if $X$ is a GL-space, then so are $C(K,X)$ (\cite[Theorem 2.10]{THL}) and $L_1(\mu,X)$ (\cite[Theorem 5.1]{H2}). We did not know whether this is true for $L_\infty(\mu, X)$ nor if $X$ is a GL-space whenever $ C(K,X)$, $L_1(\mu,X)$ or $L_\infty(\mu,X)$ is in general.


\begin{thebibliography}{99}
\bibitem{BF} J. Becerra Guerrero, M. Cueto, F.J. Fern\'{a}ndez-Polo and A.M. Peralta, On the extension
of isometries between the unit spheres of a JBW*-triple and a Banach space, J. Math. Anal. Appl. 466 (1) (2018), 127--143.


\bibitem{KVM} K. Boyko, V. Kadets, M. Mart\'{i}n, and D. Werner, Numerical index
of Banach spaces and duality,  Math. Proc. Cambridge Philos. Soc. 142 (2007),
93--102.

\bibitem{C} J. Cabello-S\'{a}nchez, A reflection on Tingley's problem and some applications, J.
Math. Anal. Appl. 476(2) (2019), 319--336.

\bibitem{CD}L. Cheng and Y. Dong, On a generalized Mazur-Ulam question: extension of isometries between unit pheres of Banach space, J. Math. Anal. Appl. 377 (2011), 464--470.

\bibitem{CA} M. Cueto-Avellaneda and A.M. Peralta, On the Mazur-Ulam property for the
space of Hilbert-space-valued continuous functions, J. Math. Anal. Appl. 479(1)
(2019), 875--902.

\bibitem{Die} J. Diestel and J.J. Uhl, \emph{Vector measures}, Math. Surveys, no. 15, Amer. Math. Soc., Providence, R.I., 1977.

\bibitem{F1} F. J. Fern\'{a}ndez-Polo and A.M. Peralta, Tingley's problem through the facial structure of an atomic
JBW*-triple, J. Math. Anal. Appl. 455  (2017), 750--760.

\bibitem{F2} F. J. Fern\'{a}ndez-Polo and A.M. Peralta, Low rank compact operators and Tingley's
problem, Adv. Math. 338 (2018), 1--40.

\bibitem{H} J.D. Hardtke, Some remarks on generalised lush spaces, Studia Math. 231
(2015), No. 1, 29--44.

\bibitem{H2} J.-D. Hardtke, On certain geometric properties in Banach spaces of vector-valued
functions, to appear in the Journal of Mathematical Physics, Analysis, Geometry;

\bibitem{F} M. Fabian, P. Habala, P. H\'{a}jek, V. Montesinos, J. Pelant, V. Zizler, Functional Analysis and Infinite Dimensional Geometry, Canad. Math. Soc. Books in Math., vol. 8, Springer-Verlag, New York (2001).

\bibitem{KM} V. Kadets and M. Mart\'{i}n, Extension of isometries between unit spheres of finite-dimensional polyhedral Banach spaces, J. Math. Anal. Appl. 396 (2) (2012), 441--447.

\bibitem{MRK} V. Kadets, M. Mart\'{i}n, J. Mer\'{i}, and A. P\'{e}rez, Spear operators between Banach spaces; Lecture Notes in Math. 2205, Springer, 2018.

\bibitem{KP} O.F.K. Kalenda and A.M. Peralta, Extension of isometries from the unit sphere
of a rank-2 Cartan factor, preprint 2019. arXiv:1907.00575.

\bibitem{Pe18} A.M. Peralta, A survey on Tingley's problem for operator algebras, Acta Sci. Math. (Szeged) 84 (1-2) (2018), 81--123.

\bibitem{THL} D. Tan, X. Huang and R. Liu, Generalized-lush spaces and the Mazur-Ulam property,
Studia Math. 219 (2) (2013), 139--153.



\bibitem{3}D. Tingley, Isometries of the unit sphere, Geom. Dedicata, 22 (1987), 371--378.

\bibitem{WH} R. Wang and X. Huang, The Mazur-Ulam property for two-dimensional
somewhere-flat spaces, Linear Algebra Appl. 562 (2019), 55--62.
\end{thebibliography}
\end{document}